\documentclass[a4paper,11pt]{article}
\usepackage{fullpage} 

\usepackage{caption}
\usepackage{subcaption}
\usepackage{booktabs,siunitx}

\usepackage{amsmath}
\usepackage{amssymb}
\usepackage{amsthm}
\usepackage{mathtools}
\usepackage{mathdots}
\usepackage{appendix}
\usepackage{graphicx}
\usepackage{enumerate}
\usepackage{microtype}
\usepackage{mleftright}
\usepackage{mdframed}
\usepackage{authblk}
\usepackage{tcolorbox}
\usepackage[english]{babel}
\usepackage[utf8]{inputenc}
\usepackage{algorithm}
\usepackage[noend]{algpseudocode}
\usepackage{xspace}
\usepackage{tikz}
\usepackage{pgfplots}
\pgfplotsset{width=8cm,compat=1.9}
\usepackage{mathdots}
\usepackage{url}
\usepackage{forest}
\forestset{default preamble={for tree={s sep+=1cm}}}

\tcbset{
    rounded corners,
    colback = white,
    before skip = 0.2cm,
    after skip = 0.5cm,
    boxrule = 1.5pt,
    arc = 5pt
}

\makeatletter
\renewcommand{\ALG@name}{Protocol}
\makeatother
\allowdisplaybreaks

\graphicspath{ {./figures/} }

\usepackage{algorithm}
\usepackage[noend]{algpseudocode}

\newtheorem{theorem}{Theorem}[section]
\newtheorem*{theorem*}{Theorem}

\newtheorem{lemma}[theorem]{Lemma}
\newtheorem{corollary}[theorem]{Corollary}
\newtheorem*{corollary*}{Corollary}
\newtheorem{conjecture}[theorem]{Conjecture}
\newtheorem{observation}[theorem]{Observation}

\theoremstyle{definition}
\newtheorem{definition}[theorem]{Definition}

\newcommand{\eps}{\varepsilon}

\newif\ifstoc
\stocfalse

\begin{document}
\title{Diagonalization Games}

\author[1,5]{Noga Alon} 
\author[6]{Olivier Bousquet}
\author[4]{Kasper Green Larsen}
\author[2,3,6]{\\Shay Moran}
\author[2]{Shlomo Moran}
\affil[1]{Departments of Mathematics and Computer Science, Tel Aviv University}
\affil[2]{Department of Computer Science, Technion, Israel}
\affil[3]{Department of Mathematics, Technion, Israel}
\affil[4]{Department of Computer Science, Aarhus University}
\affil[5]{Department of Mathematics, Princeton University}
\affil[6]{Google Research}

\maketitle

\begin{abstract}
We study several variants of a combinatorial game which is based on Cantor's diagonal argument.
    The game is between two players called Kronecker and Cantor. 
    The names of the players are motivated by the known fact 
    that Leopold Kronecker did not appreciate Georg Cantor’s arguments about the infinite, 
    and even referred to him as a ``scientific charlatan''. 

In the game Kronecker maintains a list of $m$ binary vectors, 
    each of length $n$, and Cantor's goal is to produce a new binary vector 
    which is different from each of Kronecker's vectors, or prove that no such vector exists.
    Cantor does not see Kronecker's vectors but he is allowed to ask queries of the form 
    \begin{center}
    ``\emph{What is bit number $j$ of vector number~$i$?}''
    \end{center}
    What is the minimal number of queries with which Cantor can achieve his goal?
    How much better can Cantor do if he is allowed to pick his queries \emph{adaptively}, based on Kronecker's previous replies?

The case when $m=n$ is solved by diagonalization using $n$ (non-adaptive) queries.
    We study this game more generally, and prove an optimal bound in the adaptive case and nearly tight upper 
    and lower bounds in the non-adaptive case.


    

\end{abstract}

\section{Introduction}

The concept of infinity has been fascinating philosophers and scientists for hundreds, perhaps thousands of years.
    The work of Georg Cantor (1845 -- 1918) played a pivotal role in the mathematical treatment of the infinite.
    Cantor's work is based on a simple notion which asserts that two (possibly infinite) sets 
    have the same size whenever their elements can be paired in one-to-one correspondence with each other \cite{Cantor1874}.
    Despite being simple, this notion has counter-intuitive implications:
    for example, a set can have the same size as a proper subset of it\footnote{E.g.\ the natural numbers and the even numbers, via the correspondence ``$n\mapsto 2n$''.}; this phenomena is nicely illustrated by \emph{Hilbert's paradox of the Grand Hotel}, see e.g.~\cite{enwikiHilbert}.
    
This simple notion led Cantor to develop his theory of sets, which forms the basis of modern mathematics.
    Alas, Cantor's set theory was controversial at the start, and only later became widely accepted:
\begin{center}
    \emph{The objections to Cantor's work were occasionally fierce: Leopold Kronecker's public opposition and personal attacks included describing Cantor as a "scientific charlatan", a "renegade" and a "corrupter of youth". 
    Kronecker objected to Cantor's proofs that the algebraic numbers are countable, and that the transcendental numbers are uncountable, results now included in a standard mathematics curriculum.~\cite{enwiki}}
\end{center}

\begin{figure}
\centering
\begin{minipage}{.5\textwidth}
  \centering
  \includegraphics[width=.5\linewidth]{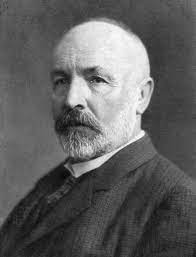}
  \caption{Georg Cantor (1845 -- 1918)}
\end{minipage}%
\begin{minipage}{.5\textwidth}
  \centering
  \includegraphics[width=.5\linewidth]{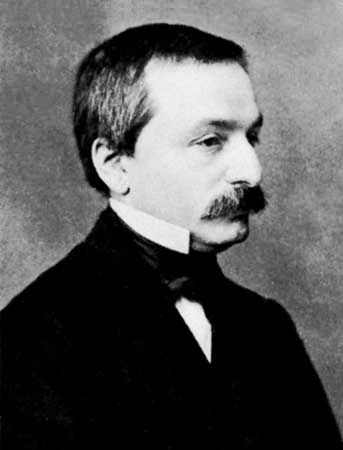}
  \captionof{figure}{Leopold Kronecker (1823 -- 1891)}
\end{minipage}
\end{figure}

\subsection{Diagonalization}
One of the most basic and compelling results in set theory is that not all infinite sets have the same size.
    To prove this result, Cantor came up with a beautiful argument, called diagonalization. 
    This argument is routinely taught in introductory classes to mathematics,
    and is typically presented as follows.
    Let $\mathbb{N}$ denote the set of natural numbers  
    and let $\{0,1\}^\mathbb{N}$ denote the set of all infinite binary vectors. 
    Clearly both sets are infinite, but it turns out that they do not have the same size:
    assume towards contradiction that there is a one-to-one correspondence $j\mapsto v_j$,
    where $v_j = (v_j(1),v_j(2),\ldots)$ is the infinite binary vector corresponding to $j\in\mathbb{N}$.
    Define a vector
    \[u = ( 1-v_1(1), 1-v_2(2), \ldots).\]
    That is, $u$ is formed by letting its $j$'th entry be equal to the negation of the $j$'th entry of $v_j$.

Notice that this way the resulting vector $u$ disagrees with $v_j$ on the $j$'th entry, 
    and hence $u \neq v_j$ for all $j$.
    Thus, we obtain a binary vector which does not correspond to any of the natural numbers via the assumed correspondence -- a contradiction.

\smallskip

Rather than reaching a contradiction, it is instructive to take a positivist perspective
    according to which diagonalization can be seen as a constructive procedure that does the following:
    \begin{tcolorbox}
    \begin{center}
        \emph{Given binary vectors $v_1, v_2,\ldots$,
        find a binary vector $u$ such that $u \neq v_j$ for all $j$.}
    \end{center}
    \end{tcolorbox}
    Moreover, notice that Cantor's diagonal argument involves querying only 
        a single entry per each of the input vectors $v_j$ (i.e.\ the ``diagonal'' entries $v_j (j)$).
        Thus, it is possible to construct $u$ while using only a little information about the input vectors $v_i$'s (a single bit per vector).

    \medskip

In this manuscript we study a finite variant of the problem in which $m$ binary vectors 
    $v_1,\ldots, v_m$ of length $n$ are given and the goal is to produce a vector $u$
    which is different from all of the $v_i$'s, or to report that no such vector exists, 
    while querying as few as possible entries of the $v_i$'s.
    We first study the case when $m<2^n$ whence such a $u$ is guaranteed to exist,
    and the goal boils down to finding one, and later the case when $m\geq 2^n$.





\section{The Cantor-Kronecker Game}


  Consider a game between two players called Kronecker and Cantor.  
  In the game there are two parameters $m$ and $n$, 
    where $m$, $n$ are positive integers. 
    Kronecker maintains a set $V = \{v_1, v_2,\ldots  , v_m\}$ of $m$ binary vectors, each of length $n$. 
    Cantor's goal is to produce a binary vector $u$, also of length $n$, which differs from each $v_i$, 
    or to report that no such vector exists.
    To do so, he is allowed to ask queries, where each query is of the form 
    \begin{center}
        ``What is bit number $j$ of vector number $i$?'',
    \end{center}
    where  $1 \leq j \leq n$, $1 \leq i \leq m$.
    Kronecker is answering each query being asked. 
    The objective of Cantor is to minimize the number of queries enabling him to produce $u$, whereas Kronecker tries to maximize the number of queries. 
    We distinguish between two versions of the game:
    \begin{itemize}
        \item In the \emph{adaptive} version Cantor presents his queries to Kronecker in a sequential manner, and may decide on the next query as a function of Kronecker's answers to the previous ones.
        \item In the \emph{oblivious} version Cantor must declare all of his queries in advance, before getting answers to any of them.  
    \end{itemize}
    For $m\leq n$ the smallest number of queries, both in the adaptive and oblivious versions, is $m$. 
    Indeed, Cantor can query bit number $i$ of $v_i$ for all $1 \leq  i \leq  m$ and return a vector $u$ whose $i$'th bit differs from the $i$'th bit of $v_i$, for all $i$.  
    The lower bound is even simpler: if Cantor asks less than $m$ queries then there is some vector $v_i$ 
    about which he has no information at the end of the game. 
    In this case he cannot ensure that his vector $u$ will not be equal to this $v_i$.

\paragraph{Note.}
   After the completion of this paper, Nikhil Vyas and Ryan Williams informed us that related diagonalization tasks have been studied in the past, both in learning theory by~\cite{BF72}, and later in circuit complexity theory. For example, in~\cite{KANNAN1982}, Kannan employs a voting technique to find a binary $n$-vector that is distinct from all vectors in a given list. More recently, and independently of our work, Vyas and Williams have studied a variant of the Cantor-Kronecker game  for the case when $m<2^n$~\cite{VW23}.
    Vyas and Williams paper focuses on methods for minimizing the number of queries needed to determine the $i$'th bit of a specific missing vector, and use them to derive lower bounds in circuit complexity. Most related to this work is their Theorem 18 and Remark 19, where they provide upper and lower bounds for the adaptive case which are tight within a multiplicative factor of $2$. (Theorem~\ref{thm:adaptive} below closes this gap.)

\paragraph{Organization.} 
We begin with the case where $m<2^n$: in the next section (Section~\ref{sec:small}) we derive nearly tight bounds both in the adaptive and oblivious cases. We do so by exhibiting and analyzing near optimal strategies for Cantor.
Then, in Section~\ref{sec:large} we consider the case where $m\geq 2^n$ and derive an optimal bound of $m\cdot n$ in this case (for both  the oblivious and the adaptive versions). We do so by exhibiting and analyzing an optimal strategy for Kronecker. Finally, in Section~\ref{sec:future} we    {discuss some algorithmic aspects}, and conclude with some suggestions to future research.

\begin{figure}
      \centering
\begin{tabular}{ c c }
  \toprule
  \vspace{-2mm}
  & $v_1 = \textcolor{red}{0}, 1,  1,  0,  1,  0$ \\
  \vspace{-2mm}
  & $v_2 = 1, \textcolor{red}{0},  0,  1,  1,  1$ \\
  \vspace{-2mm}
  & $v_3 = 1,  1,  \textcolor{red}{1},  0,  0, 0$ \\
  \vspace{-2mm}
  & $v_4 = 0,  1,  0,  \textcolor{red}{1},  1, 0$ \\
  \vspace{-2mm}
  & $v_5 = 1,  1,  0,  1,  \textcolor{red}{0}, 1$ \\
  & $v_6 = 0,  1,  1,  1,  1, \textcolor{red}{1}$ \\
  \midrule
  & $u = \textcolor{blue}{1}, \textcolor{blue}{1}, \textcolor{blue}{0}, \textcolor{blue}{0}, \textcolor{blue}{1}, \textcolor{blue}{0}$\\
  \bottomrule
\end{tabular}
    \caption{An illustration of Cantor's diagonalization: the vector $u$ at the bottom is not equal to any of the $v_i$'s at the top.}
    \label{fig:my_label}
\end{figure}

\section{The Cantor-Kronecker Game with $m<2^n$}\label{sec:small}

\subsection{Adaptive Version}

\begin{theorem}\label{thm:adaptive}
Let $g(n,m)$ denote  the smallest number of queries that suffices for Cantor when he is allowed to use adaptive strategies.
Then,
\[
g(n,m)=
\begin{cases}
m &m\leq n, \\
2m - n & n < m < 2^n.
\end{cases}
\]
\end{theorem}

The case $1 \leq m \leq n$ is proved in the previous section so we assume $n \leq m < 2^n $.

\paragraph{Upper Bound.}
We present a strategy for Cantor which combines diagonalization with another simple idea. To illustrate this idea let us first consider the case $m=n+1$. This special case appeared as a question in the \emph{2022 Grossman Math Olympiad} for high-school students, and so perhaps the reader might enjoy trying to solve it before continuing reading.

Let $v_1,\ldots, v_{n+1}$ be the input vectors.
    Cantor begins with querying the first bit of $v_1,v_2,$ and of $v_3$.
    Getting the answers, there is a bit $\eps$ so that at least two vectors among $v_1,v_2,v_3$
    have their first bit equals to $\eps$. 
    Cantor now defines the first bit of $u$ to be $u(1)=1-\eps$
    and can remove the two vectors among $v_1,v_2,v_3$ whose first bit equals $\eps$.
    Now Cantor is left with at most $n-1$ vectors and can therefore
    set the last $n-1$ coordinates of $u$ according to the diagonalization construction.

\smallskip

The general case is handled similarly by induction on $n$: 
    for $n=1$ since $n\leq m < 2^n$, also $m$ must be $1$ and the result is trivial.

Assuming the result for $n-1$, let $v_1,\ldots, v_m$ be the $m$ vectors of Kronecker.
    First,  note that there is an integer $x$ satisfying $1 \leq x \leq \lceil m/2 \rceil$ so that $n-1 \leq m-x <2^{n-1}$:
    {e.g., for $m\in [n+1, 2^{n-1}]$  let $ x= 1$ (thus $m-x=m-1$), and for  $m\in [2^{n-1}+1,2^n-1]$ let $x=m-2^{n-1}+1$ (thus $m-x=2^{n-1}-1$).
    }
    
    Having $x$ as above, Cantor first queries the first bit of each of the vectors $v_1,v_2, \ldots ,v_{2x-1}$. 
    (Note that $2x-1 \leq m$ hence this is possible). 
    Getting the answers, there is a bit $\eps\in \{0,1\}$ 
    so that at least $x$ of the vectors have their first bit equal to $\eps$. 
    Cantor now defines the first bit of his vector~$u$ to be $1-\eps$, 
    removes from the set $V$ exactly $x$ of the vectors whose first bit is $\eps$, 
    and defines as $V'$ the set of all restrictions of the remaining $m-x$ vectors to their last $n-1$ coordinates.  
    Note that $n-1 \leq m-x <2^{n-1}$. 

By the induction hypothesis, 
    Cantor can now play the game for the set $V'$ producing an appropriate vector $u'$ 
    by asking at most $2(m-x)-(n-1)$ additional queries. 
    The total number of queries is thus $(2x-1)+2(m-x)-(n-1)= 2m-n$, as needed. 
    The vector $u$ obtained by concatenating the $1$-bit vector $1-\eps$ and the vector $u'$ is
    clearly different from each member of~$V$. 
    This completes the induction step argument and finishes the proof of the upper bound.


\paragraph{Lower Bound.}
For the lower bound, we present a strategy for Kronecker which essentially mirrors Cantor's strategy from the upper bound.
    Suppose Cantor manages to produce the required vector $u$ after making exactly $b_j$ 
    queries in coordinate number $j$ of some of the vectors~$v_i$. 
    Kronecker chooses his answers ensuring that for each such $j$, the answers for bits in the $j$'th location are balanced, that is, at most $\lceil b_j/2\rceil$ of the answers are $0$ and at most $\lceil b_j/2\rceil$ of the answers are $1$. 

Consider the vector $u$ produced by Cantor. 
    For every $1 \leq  j \leq n$, there are at most $\lceil b_j/2\rceil$ vectors $v_i$ 
    known to be different than $u$ in coordinate number $j$. 
    Thus altogether there are at most
\[\sum_{j=1}^n \Bigl\lceil\frac{b_j}{2}\Bigr\rceil \leq \sum_{j=1}^n \frac{b_j+1}{2}. \]
     vectors $v_i$ that are known to Cantor to be different than $u$. 
     In order to ensure $u$ is indeed different from each $v_i$ this number has to be at least $m$ and hence
     \[m\leq \sum_{j=1}^n \frac{b_j+1}{2}.\]
     By rearranging, this implies that the total number of queries $\sum_{j=1}^{n}b_j$ 
     must be at least $2m-n$, as stated.
$\qedsymbol{}$

\subsection{Oblivious Version}

\begin{theorem}
Let $f(n,m)$ denote the  smallest number of queries that suffices for Cantor when he is restricted to use oblivious strategies. 
    Then, 
   
 \[
 f(n,m) = \begin{cases}
 m   &m \leq n\\
 m\Bigl(\log\bigl\lceil\frac{m}{n}\bigr\rceil +  o\bigl(\log\bigl\lceil\frac{m}{n}\bigr\rceil\bigr)\bigr) &n < m < 2^n.
 \end{cases}
 \]
Quantitatively, for all $n< m< 2^n$ 
\[
m\cdot\Bigl(\log\Bigl(\frac{m}{n-\log m + 1}\Bigr) - 1 \Bigr) \leq f(n,m)\leq
%
  m\Bigl\lceil  \log\big( \frac{2m}{n} \bigr)+2\log\bigl( \log\bigr(\frac{2m}{n}\bigl) \bigr)+1\Bigr\rceil,
\]
\end{theorem}
\noindent
The case $1 \leq m \leq n$ is proved above so we assume $n < m < 2^n $.

\paragraph{Upper Bound.}

Like in the adaptive case, we present a strategy for Cantor which combines diagonalization with another simple idea.
    We first illustrate this idea by handling the case~$m=n+1$,
    and again, we encourage the reader to try and handle this case before continuing reading.

Let $v_1,\ldots, v_{n+1}$ be the input vectors.
    Cantor begins with querying the first two bits of each of $v_1,v_2,$ and $v_3$ (for a total of 6 queries).
    Notice that there are $2^2=4$ possible combinations of $0/1$ patterns on the first two bits,
    but at most three of them are realized by $v_1,v_2,v_3$.
    Hence, there must be a pair of bits $\eps_1,\eps_2$ which is not realized by $v_1,v_2,$ nor $v_3$:
    \[(\eps_1,\eps_2)\notin\Bigl\{\bigl(v_1(1),v_1(2)\bigr),\bigl(v_2(1),v_2(2)\bigr),\bigl(v_3(1),v_3(2)\bigr)\Bigr\}.\]
    Thus, by setting $u(1)=\eps_1$ and $u(2)=\eps_2$, Cantor rules out $v_1,v_2,v_3$
    and is left with $n-2$ vectors $v_3,\ldots, v_{n+1}$ which can be obliviously ruled out with the last $n-2$ using 
    diagonalization. 
    
\smallskip

For the general case, let $d$ be an integer (to be determined later).
    Pick mutually disjoint subsets of coordinates $J_1,\ldots, J_{\lfloor n/d\rfloor}\subseteq [n]$, each of size $d$,
    and pick a partition of the $m$ vectors to $\lfloor n/d\rfloor $ subsets $V_1,\ldots, V_{\lfloor n/d\rfloor}$
    such that the partition is as balanced as possible (i.e.\ the difference between each pair of sizes is $\leq 1$).
    Thus, each set has size
    \[\lvert V_i\rvert \leq \Bigl\lceil \frac{m}{\lfloor n/d \rfloor} \Bigr\rceil\leq \frac{2md}{n}. \]
    Cantor queries (obliviously) as follows.
    \begin{center}
    \emph{For each $i$ and each vector in $V_i$ query all the coordinates in $J_i$.}
    \end{center}
    Thus, the total number of queries is exactly $m\cdot d$. 
    Now, notice that if $d$ satisfies
    \begin{equation}\label{eq:ob1}
    2^{d} >  \frac{2md}{n},
    \end{equation}
    then there must exist an assignment $f_i:J_i\to \{0,1\}$
    such that $f_i$ disagrees with each of the vectors in $V_i$ on at least one coordinate in $J_i$.
    Hence Cantor can output the vector $u$, which agrees with each of the $f_i$ on $J_i$.
   Note that  Equation~\ref{eq:ob1} is satisfied iff
     $\frac{2^d}{d} >  \frac{2m}{n}$; 
     since $m>n$, it can be verified    that this inequality holds when
     $d\geq \log( \frac{2m}{n} )+2\log( \log(\frac{2m}{n}) )+1 $.
     Thus for   
     $d=\Bigl\lceil  \log\big( \frac{2m}{n} \bigr)+2\log\bigl( \log\bigr(\frac{2m}{n}\bigl) \bigr)+1\Bigr\rceil $, 
     the total number of queries is at most
   
 \[
 m\cdot d =  m\Bigl\lceil  \log\big( \frac{2m}{n} \bigr)+2\log\bigl( \log\bigr(\frac{2m}{n}\bigl) \bigr)+1\Bigr\rceil.
 \]
 \noindent
  {$\bigl[$A partition similar to the above  is used in Theorem 18 of \cite{VW23} in an adaptive algorithm which aims at minimizing the number of queries required to retrieve any single bit of a specific missing vector.$\bigr]$}

\paragraph{Lower Bound.}

The lower bound proof is based on the following simple idea.
    Let $J_i$ denote the set of coordinates of $v_i$ which Cantor queries.
    Thus, the total number of queries Cantor uses is $\lvert J_1\rvert + \ldots + \lvert J_m\rvert$.
    Now, let $f_i:J_i\to\{0,1\}$ 
    denote Kronecker's answers for the queries on $v_i$.
    The crucial observation is that the vector $u$ that Cantor outputs must satisfy
    \[(\forall i): u\vert_{J_i}\neq f_i.\]
    Indeed, if $u\vert_{J_i} = f_i$ for some $i$ then Kronecker can fail Cantor by picking his $i$'th vector $v_i$ to be equal to Cantor's output $u$
    (which would be consistent with Kronecker's answers).

We summarize the above consideration with a definition that characterizes the winning (or losing) strategies of Cantor in the oblivious case.
\begin{definition}[Covering Assignments]\label{def:covering}
We say that a sequence of sets $J_1,\ldots, J_m\subseteq [n]$ \emph{has a covering assignment}
if there are $m$ functions $f_i:J_i\to\{0,1\}$ such that every binary vector~$v\in\{0,1\}^n$
agrees with one of the $f_i$ on $J_i$ (i.e.\ $v\vert_{J_i}=f_i$). 
\end{definition}

Thus, Kronecker has a winning strategy if and only if the sequence of sets $J_1,\ldots, J_m$ that Cantor queries
    has a covering assignment. The following lemma establishes the lower bound.
    \begin{lemma}
    Let $J_1,\ldots, J_m\subseteq [n]$ such that 
    \begin{equation}\label{eq:size}
    \lvert J_1\rvert + \ldots +\lvert J_m\rvert < m\cdot\Bigl(\log\bigl(\frac{m}{n-\log m + 1}\bigr) - 1 \Bigr).  
    \end{equation}
    Then, $J_1,\ldots, J_m$ has a covering assignment.

Equivalently, if for each vector $v_i$ Cantor queries its entries in $J_i$ and Equation~\ref{eq:size} holds,
then Kronecker has a winning strategy.
    \end{lemma}
\begin{proof}

Let $t_i=\lvert J_i\rvert$ and let $t=\sum_i t_i$.
    Assume, without loss of generality, that $t_1 \leq t_2 \leq \ldots \leq t_m$. 
    To prove a lower bound of the form $md$ for $t$, where $d$ will be specified later,
    we show that if $t$ is smaller than $md$ 
    then there are $m$ functions $f_i:J_i\to\{0,1\}$ 
    so that for every possible vector $v \in \{0,1\}^n$ there is $i\leq m$ so that $v\vert_{J_i} = f_i$. 

We do so by explicitly constructing the $f_i$'s 
    (which corresponds to describing a winning strategy for Kronecker). 
    Starting with the set $V=\{0,1\}^n$ of all possible potential vectors $z$, go over the vectors $v_i$ in order. 
    In step $i$ we choose the function $f_i:J_i\to\{0,1\}$ such that $\lvert \{v\in V: v\vert_{J_i} = f_i\}\rvert$ is maximized. 
    Since there are $2^{t_i}$ possible choices for $f_i$, the maximizing choice satisfies
    \[\Bigl\lvert \{v\in V: v\vert_{J_i} = f_i\}\Bigr\rvert \geq \frac{\lvert V\rvert}{2^{t_i}}.\]
    After picking $f_i$, we remove all the vectors of $V$ that agree with $f_i$ and proceed to the next step. 
    Therefore, after the first $i$ steps, the size of the set $V$ of the remaining vectors is at most 
    \[2^n\prod_{j=1}^i (1-1/2^{t_j}).\]

We can continue with this analysis until the size of the set $V$ becomes smaller than $1$, namely the set becomes empty. 
    It is a bit better, however, to apply a simpler reasoning once the size of $V$ becomes smaller than $2^d$, 
    and only argue that at least one vector from $V$ is eliminated in each step.
    (Continuing the same analysis as before would only guarantee that $V$ shrinks 
    by a factor  of $(1-1/2^{t_i})$ which by the choice of $d$ would be roughly $1-1/2^d < 1$).  
    To simplify the computation it is not too wasteful to apply the simpler analysis 
    already when the size of $V$ becomes smaller than $m/2$. 
    If this happens in the first $m/2$ steps then by removing a single vector
    in each of the remaining steps we will eliminate all of the vectors.
    This means that if
\[
2^n \prod_{j=1}^{m/2} \bigl(1-1/2^{t_j}\bigr) \leq \frac{m}{2}
\]
then the sequence $J_1.\ldots,J_m$ has a covering assignment. 
    Since $d$ is such that the total number of queries is $m\cdot d$,
    the above amounts to $\sum_{j=1}^{m/2} t_j \leq md/2$;
    that is, the average $t_j$ for $1 \leq j \leq m/2$ is at most $d$. 
    This implies that 
\begin{align*}
2^n \prod_{j=1}^{\frac{m}{2}} \bigl(1-\frac{1}{2^{t_j}}\bigr) &\leq 2^n \prod_{j=1}^{\frac{m}{2}} \exp\Bigl(-\frac{1}{2^{t_j}}\Bigr) \tag{$1+x \leq \exp(x)$ for all $x\in\mathbb{R}$}\\
                                                              &= 2^n \exp\Bigl(-\sum_{j=1}^{\frac{m}{2}}\frac{1}{2^{t_j}}\Bigr)\\
                                                              &\leq 2^n \exp\Bigl(-\frac{m}{2^{d+1}}\Bigr),
\end{align*}
where the last inequality follows because $\exp(-x)$ is decreasing and because  
\[
\sum_{j=1}^{\frac{m}{2}}\frac{1}{2^{t_j}} \geq \frac{m}{2}\cdot \frac{1}{2^{\frac{1}{m/2}\sum_{j=1}^{m/2}t_j}} \geq \frac{m}{2}\cdot \frac{1}{2^d},
\]
which follows by convexity of the function $f(x)=2^x$ and because $t_1\leq t_2\leq\ldots \leq t_m$.

We have thus shown that if $\lvert J_1\rvert + \ldots +\lvert J_m\rvert = m\cdot d$ such that
\[
2^n \exp\Bigl(-\frac{m}{2^{d+1}}\Bigr) \leq \frac{m}{2}
\]
then the sequence $J_1,\ldots, J_m$ has a covering assignment. 
The last inequality surely holds provided
\[
\frac{m}{2^{d+1}}  \geq n+1-\log m.
\]
That is, provided
\[
2^{d+1} \leq \frac{m }{n+1-\log m},
\]
or
\[
d \leq \log \Bigl(\frac {m}{n+1 -\log m}\Bigr)-1
\]
completing the proof.  
\end{proof}

\section{The Cantor-Kronecker Game with $m\geq 2^n$ }\label{sec:large}

Assume now that Kronecker's list $V$ consists of $m\geq 2^n$ binary vectors of length $n$.
    In this case $V$ may contain all the binary vectors of length $n$ and there is no vector Cantor can output that is different from each vector on Kronecker's list.    
     In this regime it is more natural to first focus on the decision problem in which Cantor's goal is to decide whether $V$ contains $\{0,1\}^n$, and if this is not the case, to provide a vector which is not in $V$.\footnote{{Later we will see that the decision and search variants are in fact equivalent.}} 
    Clearly Cantor can achieve this if he queries all $mn$ possible queries. Can he do better? 

We first observe that $mn$ queries are in fact needed in the oblivious case: 
    assume that Cantor submits only $mn-1$ queries, and leaves the $j$'th bit of $v_i$ unqueried. 
    Then Kronecker may set~$v_i$ to be the unique occurrence of the all ones vector $1^n$, and set the remaining $m-1$ vectors in $V$ to include all $2^n-1$ vectors that are different from the all ones vector. 
    Clearly, it is necessary for Cantor to query also the last bit of $v_i$ in order to see whether $v_i$ is the all ones vector or not.
    Consequently, Cantor must query all $mn$ queries in the oblivious case.
    
How about the adaptive case? 
    A similar argument shows that for $m=2^n$, Kronecker can force $mn=2^n n$ queries also in the adaptive case, 
    by using a list which contains each binary vector of length $n$ exactly once: indeed, if only $mn-1$ bits are queried, then the last, yet unqueried bit, belongs to a vector which occurs only once in $V$. Hence it is necessary to get the value of this bit in order to verify that $V$ contains all $2^n$ vectors. 

The case when $m>2^n$ turns out to be more subtle. 
    Nevertheless, we prove that $mn$ queries are necessary even in this case. We start with introducing some notation.

\paragraph{Notation.}
Each step of the game consists of a query by Cantor followed by a response by Kronecker. The status of the game after each such step is given by an $m\times n$ matrix  $L$, where $L(i,j)$ denotes the status of the $j$'th bit of $v_i$, that is: $L(i,j)\in \{ 0,1,\star\}$, where
$L(i,j)=\star$ means that the $j$'th bit of $v_i$ was not queried yet, and otherwise $L(i,j)$ equals the value of this bit as answered by Kronecker.
\begin{definition}
$\mathtt{FIXED}(L) =\bigl\{v\in L: v\in\{0,1\}^n\bigr\}$.
That is, $\mathtt{FIXED}(L)$ is the set of all vectors in $L$ that were fully queried by Cantor.
\end{definition}
\begin{definition}
  $L$ is \emph{complete} if $\mathtt{FIXED}(L) = \{0,1\}^n$.
  \end{definition}
  \begin{definition}
   A subset  $S$ of $2^n$ rows of $L$ is {\em useful} if it either contains all the $2^n$ binary vectors of length $n$, or it can be {\em converted} to this set by replacing each $\star$-entry in $S$ by $0$ or $1$.
   \end{definition}
   \begin{definition}
    A matrix $L$ is {\em unblocked} if it can be completed; that is, if $L$ has a useful subset. Otherwise $L$  is called  {\em blocked}.
   \end{definition}

 Notice that for $m\geq 2^n$, the $m$ by $n$ matrix all whose entries are $\star$ is unblocked.
 
 As a warmup, and to get used to the definitions, let us assume first that Cantor's queries the vectors one by one according to their order; i.e.\ he first queries all the bits of $v_1$ from left to right, then all the bits of $v_2$ from left to right, and so on. 
 We use the following strategy for Kronecker:
 when Cantor queries the $j$'th bit of $v_i$ (i.e.\ the value of $L(i,j)$), Kronecker replies according to the following  ``$0$ first" strategy:
 \begin{equation}\label{str:CK}
\text{modified value of }L(i,j) = \begin{cases}
  1&\mbox{If setting }L(i,j)   \mbox{ to 0 blocks }$L$\\
0&\mbox{otherwise}
\end{cases}
\end{equation}
 
 It is not hard to verify that since Cantor queries the vectors one by one, and from left (most significant bit) to right, the following matrix is produced: each of the first $m-2^n+1$ rows will be set to the all-zeros vector, and the last $2^n-1$ rows will be set to the $2^n-1$ non zero vectors in increasing lexicographical order: starting with $0^{n-1}1$ and ending with $1^n$. Hence Cantor is forced to query all $mn$ entries as in the oblivious case.
 
 Interestingly, it turns out that, for {\em any} strategy of Cantor, the above ``0 first" strategy of Kronecker forces Cantor to make $mn$ queries.

\begin{theorem}\label{thm:KC} Let $m>2^n$. Then for any strategy of Cantor, the ``0 first" strategy of
Kronecker forces Cantor to make $mn$ queries in order to determine if $L$ contains $\{0,1\}^n$.
\end{theorem}

In the following we consider an arbitrary execution of the game, where Kronecker follows the ``0 first" strategy (and Cantor's strategy is arbitrary). 
We denote by $L_t$ the $m\times n$ matrix $L$ after $t$ steps of the game; thus $L_0$ is the initial matrix which is filled only with $\star$'s.

By the fact that if $L$ is unblocked and $L(i,j)=\star$, then it is possible to set $L(i,j)$ to $0$ or to~$1$ without blocking $L$, we get:
\begin{observation}\label{obs:complete}
 If $L_t$ is unblocked, so is $L_{t+1}$. Hence $L_{mn}$ is complete; i.e.\ it contains $\{0,1\}^n$.
\end{observation}
 
\begin{definition} We say that a row $L(i)$ is {\em essential} for an unblocked matrix $L$ if every useful subset of $L$'s rows contains $L(i)$.
\end{definition}
Note that  if $L_t(i)$ is essential for $L_t$, then $L_s(i)$ is essential for $L_s$ for all $s\geq t$. Also, if $L_{mn}(i)$ is essential for $L_{mn}$, then $L_{mn}(i)$ is equal to a unique vector in $\{0,1\}^n$ which is different from all other rows of $L_{mn}$.
\begin{lemma}\label{lem:NE}
    Assume that $L_t(i)$ is not essential for $L_t$ and $L_t(i,j)=\star$. If $L_t(i,j)$ is queried at time $t+1$, then   it is set to $0$, i.e. $L_{t+1}(i,j)=0$.
\end{lemma}
\begin{proof}  By the ``0 first" strategy, and the fact that if  $L(i)$ is not essential for an unblocked matrix~$L$,  then setting $L(i,j)$ to $0$ does not block $L$.
\end{proof}
By a  straightforwards induction Lemma \ref{lem:NE} implies: 
\begin{corollary}\label{cor:4.9}
If $L_{t}(i)$ is not essential for $L_{t}$, then $L_{t}(i)$ contains no $1$'s (only $0$'s or $\star$'s). Specifically, if $L_{mn}(i)$ is not essential for $L_{mn}$, then $L_{mn}(i)$ is the zero vector $0^n$. Hence, every row of $L_{mn}$  which is not the zero vector is essential, and thus it is different from all other rows of $L_{mn}$. 
\end{corollary}
\begin{lemma}\label{lem:LAST}
  Let $L_{mn-1}(i,j)$ be the last bit queried in the game. Then $L_{mn-1}(i)$ is an essential row of $L_{mn-1}$.
  \end{lemma}
\begin{proof} To simplify notation, we assume without loss of generality that $j=1$. 
Assume towards contradiction that  $L_{mn-1}(i)$ is not essential for $L_{mn-1}$. By Corollary~\ref{cor:4.9}, this implies that $L_{mn-1}(i)=\star0^{n-1}$ and $L_{mn}(i)=0^n$.  (i.e.\ Kronecker sets $L_{mn-1}(i,1)$ to 0 at Cantor's $mn$'th query).  
Since $L_{mn}$ is complete (Observation~\ref{obs:complete}), this implies that $L_{mn-1}$ contains a distinct occurrence of each of the $2^n-1$ nonzero vectors of $\{0,1\}^n$, and in particular for some $k\neq i$, $L_{mn-1}(k)$ is the unique row of $L_{mn-1}$ which equals $10^{n-1}$. Then, any subset $S$ of $L_{mn-1}$ which contains
\begin{itemize}
\item the row $L_{mn-1}(i)$, 
\item the $2^n-2$ non zero rows of $L_{mn-1}$ excluding $L_{mn-1}(k)$, and
\item  some zero row of $L_{mn-1}$ (by Corollary~\ref{cor:4.9} there are $m-2^n>0$ such rows in $L_{mn-1}$),  
 \end{itemize}
is a useful subset of $L_{mn-1}$ which does not contain $L_{mn-1}(k)$. 
Hence $L_{mn-1}(k)$ is not essential for $L_{mn-1}$, and by Lemma \ref{lem:NE} $L_{mn-1}(1)=0\neq 1$, which stands in contradiction with  $L_{mn-1}(1)=10^{n-1}$. 
\end{proof}
\begin{proof}[Proof of Theorem \ref{thm:KC}]
Let $L_{mn-1}(i,j)$ be the last query in the game.   By Lemma \ref{lem:LAST}, $L_{mn-1}(i)$, and hence also $L_{mn}(i)$,  is essential, meaning that $L_{mn}(i)$  is different from all other rows of $L_{mn}$. Thus Cantor must get the value of $L_{mn-1}(i,j)$  in order to reach a decision.
\end{proof}

\paragraph{A remark on computational complexity.}
 A naive implementation of the ``$0$ first'' strategy might take exponential time: 
    indeed, it requires checking whether setting the queried bit to $0$ blocks the current matrix, 
    which involves checking a potentially exponential list of constraints.
Nevertheless, we next show that this strategy in fact admits a polynomial time implementation.
    Firstly, notice that the first $m-2^n$ steps are trivially efficient, 
    because setting $L(i,j)$ to any value cannot block $L$ (since at least~$2^n$ rows of $L$ are not queried yet). 

Thus it suffices to show that in each later step, deciding whether setting 
    $L(i,j)$ to $0$ blocks the matrix, can be performed in time which is polynomial in $mn$, the size of $L$.
    Let $L_t$ be the matrix $L$ after $t$ steps of the game, $t>m-2^n$. 
    Consider the bipartite graph $G_t=(A_t,B,E_t)$, 
    where~$A_t=\{L_t(i):1\leq i\leq m\}$ is the set of rows of $L_t$, $B=\{0,1\}^n$, 
    and $(L_t(i),u)\in E_t$ if and only if~$L_t(i)$ can be converted to the binary vector $u$ 
    by replacing the~$\star$'s in $L_t(i)$ (if any) by binary digits. 
    Then, a subset $S$ of $L_t$ is useful for $L_t$ if and only if 
    $G_t$ contains a perfect matching between the vertices in $A_t$ which correspond to $S$ and $B$. 

Assume now that we are given the graph $G_t$, and the corresponding matching,
    and let $L_t(i,j)$ be the entry queried by Cantor at step $t+1$. 
    To check if setting $L_t(i,j)$ to 0 blocks $L_t$, we remove from $G_t$ all the edges 
    $(L_t(i),u)$ in which $u(j)=0$, and check if the resulted graph contains a perfect matching. 
    Since we are given a perfect matching $M_t$ for $G_t$, 
    and removing these edges eliminates at most one edge from $M_t$, 
    this checking can be done by executing one phase in some classical algorithm for bipartite matching, which can be done in $O(|E_t|)=O(m2^n)=O(m^2)$ time (see e.g. \cite{EvenBook}).

\section{Concluding Remarks and Future Research}\label{sec:future}

We studied the Cantor-Kronecker game for different values of $m$ and $n$: when $m \leq n$ the trivial lower bound of $m$ is tight (a lower bound of $m$ follows because Cantor must query at least one bit in each vector); when $m\geq 2^n$, the trivial upper bound of $mn$ is tight (an upper bound of $mn$ follows because querying all the bits is clearly sufficient); when $n < m < 2^n$ the landscape is more interesting, and in particular the bounds depend on whether Cantor is adaptive or oblivious.

\paragraph{Further Research.}
We conclude with suggestions for possible future research:
\begin{enumerate}

\item Study the Cantor-Kronecker game when there are $r$ rounds of adaptivity: i.e.\ there are $r$ rounds in which Cantor can submit  queries, and in each round the submitted queries may depend on Kronecker's answers to queries from previous rounds. How does the query complexity change as a function of $r$? Note that $r=1$ is the oblivious case and $r=\infty$ is the adaptive case. (In fact $r=n$ is already equivalent to $r=\infty$.)
\item Consider the following generalization of the game. Let $k\leq m, \ell\leq n$ be positive integers. Kronecker maintains an $m\times n$ binary matrix, and Cantor queries the entries of Kronecker's matrix. Cantor's goal is to find a $k\times \ell$ matrix which does not appear as a submatrix of Kronecker's $m\times n$ matrix, or to decide that one does not exist. So, the original game is when $k=1, \ell=n$. What is the query complexity as a function of $k,\ell,m,n$ in the adaptive/oblivious case? For which values does Cantor have a strategy that uses strictly less than $m\cdot n$ queries?
    \item
   Find tighter bounds for the oblivious case. Specifically, notice that Cantor's original diagonalization provides tight bound on the number of queries needed for the oblivious case when $m\leq n$. 
   {It will be interesting to derive tight bounds and optimal strategies in the remaining cases.
   As we exemplify below, this question has connections with natural combinatorial problems.}
   
Consider the case when $m$ is at the other end of the scale, 
   namely  {$2^{n-1}\leq m<2^n$}.
   Then, Cantor can win the game by querying $nm-d$ bits, where $d=2^n-m-1$.
  {In fact, it suffices that Cantor chooses his queries such that} each of the $d$ unqueried entries belongs to a different vector: in this case any assignments of values to the unqueried entries covers (in the sense of Definition~\ref{def:covering})  the $m-d$ fully queried vectors, and at most two additional vectors per each of the remaining $d$ vectors (each of which contains one unqueried entry): altogether at most $(m-d)+2d=m+d$ vectors. Hence, Cantor is guaranteed to win the game provided that $m+d<2^n$ (equivalently $d\leq 2^n-m-1$).
  
 {Is the above strategy optimal? i.e., can Kronecker win the game when Cantor queries only $mn-(2^n-m)$  bits?   
 Informally, Kronecker has a winning strategy if, for any distribution of the $2^n-m$ unqueried entries, there is an assignment  which covers sufficiently many vectors. This is formalized below.
 \begin{definition}[$\mathtt{cube}(v),J\mbox{-}$cube]\label{def:cube}
 Let $v$ be a vector with possibly some unqueried entries.  $\mathtt{cube}(v)$ is the set of  binary vectors which can be obtained by replacing the unqueried entries in $v$ by zeros or ones. In particular, $\mathtt{cube}(v)=\{v\}$ if $v$ is fully queried.
 The cube $\mathtt{cube}(v)$ is called a $J$-cube if $J =\{j:~\mbox{the } j'th~\mbox{bit of } v ~\mbox{is not queried}\}$.
 For $j\in [n]$, a $\{j\}\mbox{-}$cube is denoted by {$j$-edge}. 
 \end{definition}
 
  {Assume that Cantor distributes the $(2^n-m)$ unqueried entries  among vectors $v_1,\ldots,v_q$. 
  Then Kronecker answers to the queried entries define a cube $C(v_i)$ for each vector $v_i$.} Kronecker wins if and only if those cubes cover $\{0,1\}^n$. Hence Kronecker has a winning strategy when Cantor uses $mn-(2^n-m)$ queries ($2^{n-1}+1\leq m <2^n$) if and only if the following holds:
\begin{conjecture}\label{con:opt}
 {Let  $d=2^n-m<2^{n-1}$.}
For any collection $J_1,J_2,\ldots,J_q$ of nonempty subsets of $[n]$ satisfying $\sum_{i=1}^q |J_i|=d$,   there are cubes $C_1,\ldots,C_q$ s.t. $C_i$ is a $J_i\mbox{-}$cube,   and $\lvert\bigcup_{i=1}^{q} C_i\rvert\geq d+q$.
\end{conjecture}
%

The following result of \cite{FHK93} proves Conjecture  \ref{con:opt} for the case that each $J_i\mbox{-}$cube is a  $j_i\mbox{-}$edge. 
    \begin{theorem}[\cite{FHK93}]
    Let $d<2^{n-1}$. For any multiset $D=\{j_1,j_2,\ldots,j_d\}$ of  elements of $[n]$,  $\{0,1\}^n$ contains a matching $\{e_1,\ldots,e_d\}$ s.t. for  $i=1,\ldots,d$, $e_i$ is a $j_i\mbox{-}$edge. 
    \end{theorem}
    It is also shown in \cite{FHK93} that Conjcture \ref{con:opt} does not hold when  $d= 2^{n-1}$: in this case   a corresponding matching exists if and only if each element in $[n]$ occurs an even number of times in $D$.
    This implies that when $m=2^{n-1}$ Cantor has a winning strategy with only $mn -(2^n-m)=mn-2^{n-1}$ queries: he may query $n-1$ entries per each vector, so that at least one dimension is left unqueried in an odd number of vectors. 
    }
    


  
 \end{enumerate}

\subsection*{Acknowledgements} We would like to thank 
Nikhil Vyas and Ryan Williams for bringing references \cite{BF72,KANNAN1982,VW23} to our attention, and Ron Holzman for informing us about the result in~\cite{FHK93}. 
We also thank Ariel Gabizon, and Yuval Wigderson for providing insightful comments on a previous version of this manuscript.
    
\bibliographystyle{alphaurl}
\bibliography{bib.bib}

\end{document}